\documentclass[12pt]{article}
\usepackage{amssymb}
\usepackage{amsthm}
\usepackage{amsmath}
\usepackage{graphicx}
\usepackage{enumerate}

\newtheorem{theorem}{Theorem}
\newtheorem{definition}[theorem]{Definition}
\newtheorem{lemma}[theorem]{Lemma}

\newtheorem{example}[theorem]{Example}

\title{How to introduce the connective implication in orthomodular posets}
\author{Ivan~Chajda and Helmut~L\"anger}
\date{}
\begin{document}
\footnotetext[1]{Support of the research by \"OAD, project CZ~02/2019, and support of the research of the first author by IGA, project P\v rF~2019~015, is gratefully acknowledged.}
\maketitle
\begin{abstract}
Since orthomodular posets serve as an algebraic axiomatization of the logic of quantum mechanics, it is a natural question how the connective of implication can be defined in this logic. It should be introduced in such a way that it is related with conjunction, i.e.\ with the partial operation meet, by means of some kind of adjointness. We present here such an implication for which a so-called unsharp residuated poset can be constructed. Then this implication is connected with the operation meet by the so-called unsharp adjointness. We prove that also conversely, under some additional assumptions, such an unsharp residuated poset can be converted into an orthomodular poset and that this assignment is nearly one-to-one. 
\end{abstract}
 


Orthomodular posets are considered as an algebraic axiomatization of the logic of quantum mechanics, see e.g.\ \cite B. At the very beginning of this theory, G.~Birkhoff and J.~von~Neumann as well as K.~Husimi considered orthomodular lattices for this reason but later on it was shown in accordance with physical theory and experiments that in this structure the existence of suprema is granted only for so-called orthogonal elements. Hence, the interest of researchers was shifted to orthomodular posets, however, orthomodular lattices remain still very important structures which have a number of interesting properties which motivate us also for analogical treaty for orthomodular posets.

On the other hand, when some algebraic structure is used as an axiomatization of a propositional logic, we must ask for a connective implication. Namely, implication is the most productive connective which enables deductive reasoning in the corresponding logic. In the case of orthomodular lattices, we usually consider the so-called Sasaki operation (see e.g.\ \cite B) for this reason. As pointed out in \cite{GJKO}, a connective implication is ``good'' if it satisfies the so-called adjointness with respect to conjunction. In other words, we need to relate our structure with a residuated one. The authors showed in \cite{CL17a} and \cite{CL17b} how implication in orthomodular lattices is derived by the Sasaki operation in order to obtain a so-called left-residuated l-groupoid. Hence, orthomodular lattices can be really considered as an axiomatization of some reasonable logic connected with the logic of quantum mechanics.

In the present paper we solve the question of finding an implication in orthomodular posets in the way that a certain residuation is possible. Because the operation meet in orthomodular posets is only partial, one cannot expect that the resulting implication will be ``sharp'', i.e.\ its values will not be elements but subsets. A similar approach was already used by the authors also for another so-called quantum structure, namely for effect algebras which are also only partial algebras. Effect algebras describe the behavior of effects in the event-state systems of the logic of quantum mechanics, see \cite{CL2}.

Although our residuated structure corresponding to an orthomodular poset can be recognized to be rather complicated, we show that also conversely, every such a structure can be transformed into an orthomodular poset.

Recall that a {\em bounded poset with an antitone involution} is an ordered quintuple $(P,\leq,{}',0,1)$ where $(P,\leq,0,1)$ is a bounded poset and $'$ is a unary operation on $P$ such that the following conditions are satisfied for all $x,y\in P$:
\begin{itemize}
\item $x\leq y$ implies $y'\leq x'$,
\item $(x')'=x$.
\end{itemize}
We say that the elements $a,b$ of $P$ are {\em orthogonal} to each other if $a\leq b'$ (or, equivalently, $b\leq a'$).

Further recall that an {\em orthomodular poset} is a bounded poset $(P,\leq,{}',0,1)$ with an antitone involution satisfying the following conditions for all $x,y\in P$:
\begin{itemize}
\item $x\vee y$ is defined provided $x\leq y'$,
\item $x\leq y$ implies $y=x\vee(y\wedge x')$.
\end{itemize}
The last condition is called the {\em orthomodular law}. Observe that in case $y=1$ this law implies $x\vee x'=1$. Since $'$ is an antitone involution this further implies $x\wedge x'=0$. Thus $'$ is a complementation.

Note that in case $x\leq y$ the expression $x\vee(y\wedge x')$ is well-defined. Of course, $x\vee y$ may exist also for elements $x$ and $y$ which are not orthogonal to each other. For example, if $x\leq y$ then $x\vee y$ exists and equals $y$. Because every poset with an antitone involution satisfies the De Morgan laws, we have that $x\wedge y$ exists provided $x'\leq y$.

Let $(P,\leq)$ be a poset, $a,b\in P$ and $A,B\subseteq P$. Put
\begin{align*}
L(A) & :=\{x\in P\mid x\leq y\text{ for all }y\in A\}, \\
U(A) & :=\{x\in P\mid y\leq x\text{ for all }y\in A\}.
\end{align*}
Instead of $L(A\cup B)$, $L(A\cup\{a\})$, $L(\{a,b\})$ and $U(L(A))$ we simply write $L(A,B)$, $L(A,a)$, $L(a,b)$ and $UL(A)$, respectively. Moreover, $A\leq B$ should mean that $x\leq y$ for all $x\in A$ and $y\in B$. Instead of $\{a\}\leq A$ and $A\leq\{a\}$ we simply write $a\leq A$ and $A\leq a$, respectively.

Let $(P,\leq,{}',0,1)$ be an orthomodular poset and $A\subseteq P$. Then we put $A':=\{x'\mid x\in A\}$. Moreover $A\vee a:=\{x\vee a\mid x\in A\}$ is defined if $A\leq a'$ or, equivalently, $a\leq A'$. Analogous statements hold for meet.

\begin{lemma}\label{lem1}
Let $(P,\leq,{}',0,1)$ be an orthomodular poset and $a,b\in P$. Then
\begin{align*}
U(a,b) & =a\vee(U(a,b)\wedge a'), \\
L(a,b) & =b\wedge(L(a,b)\vee b').
\end{align*}
\end{lemma}

\begin{proof}
Let $c\in U(a,b)$ and $d\in L(a,b)$. Since $'$ is an antitone involution of $(P,\leq)$, we have
\begin{align*}
  (x\vee y)' & =x'\wedge y'\text{ for all }x,y\in P\text{ with }x\leq y', \\
(x\wedge y)' & =x'\vee y'\text{ for all }x,y\in P\text{ with }x'\leq y.
\end{align*}
Since $a\leq c$ and $c\wedge a'\leq a'$, the expression $a\vee(c\wedge a')$ is defined and $a\vee(c\wedge a')=c$. Since $d\leq b$ and $b'\leq d\vee b'$, the expression $b\wedge(d\vee b')$ is defined and
\[
b\wedge(d\vee b')=(b'\vee(d'\wedge b))'=(d')'=d,
\]
thus the assertion is evident.
\end{proof}

In an orthomodular lattice $\mathbf L$, the connective implication can be defined in six different ways which coincide with $x'\vee y$ in the case when $\mathbf L$ becomes a Boolean algebra. However, in order to show that it is an adjoint of conjunction, only the following one
\[
x\rightarrow y:=x'\vee(x\wedge y)
\]
is suitable, see \cite{CL17b}. In analogy to this formula, we can define the connective implication in orthomodular posets $\mathbf P=(P,\leq,{}',0,1)$ as follows:
\[
x\rightarrow y:=x'\vee L(x,y).
\]
Since $L(x,y)\leq x$, this expression is everywhere defined. Of course, the results of this implication are subsets of $P$ instead of elements of $P$ as in orthomodular lattices. However, if implication is defined in this way then it has some nice and expected properties as presented in the following lemma.

\begin{lemma}
Let $\mathbf P=(P,\leq,{}',0,1)$ be an orthomodular poset and $x\rightarrow y:=x'\vee L(x,y)$ for all $x,y\in P$. Then
\begin{enumerate}[{\rm(i)}]
\item $x\rightarrow0\approx\{x'\}$,
\item $1\rightarrow x\approx L(x)$,
\item if $x\leq y$ then $x\rightarrow y=[x',1]$,
\item if $x\leq y$ then $U(x\rightarrow y)=\{1\}$,
\item $x\rightarrow x'=\{x'\}$
\end{enumerate}
{\rm(}$x,y\in P${\rm)}.
\end{lemma}

\begin{proof}
Let $a,b\in P$.
\begin{enumerate}[(i)]
\item We have $x\rightarrow0\approx x'\vee L(x,0)\approx x'\vee\{0\}\approx\{x'\}$.
\item We have $1\rightarrow x\approx1'\vee L(1,x)\approx0\vee L(x)\approx L(x)$.
\item Assume $a\leq b$. Then $a\rightarrow b=a'\vee L(a,b)=a'\vee L(a)\subseteq[a',1]$. Conversely, if $c\in[a',1]$ then $c=a'\vee(c\wedge a)\in a'\vee L(a)=a\rightarrow b$.
\item If $a\leq b$ then by (iii) we have $U(a\rightarrow b)=U([a',1])=\{1\}$.
\item We have $x\rightarrow x'\approx x'\vee L(x,x')\approx x'\vee\{0\}\approx\{x'\}$.
\end{enumerate}
\end{proof}

We are going to show that this implication is related with conjunction by means of unsharp residuation.

Recall that a {\em partial commutative monoid} is a partial algebra $(A,\odot,1)$ of type $(2,1)$ satisfying the following conditions for all $x,y,z\in A$:
\begin{itemize}
\item if $(x\odot y)\odot z$ and $x\odot(y\odot z)$ are defined then they coincide,
\item $x\odot1\approx1\odot x\approx x$,
\item $x\odot y$ is defined if and only if so is $y\odot x$ and in this case $x\odot y=y\odot x$.
\end{itemize}

If $(P,\leq,{}',0,1)$ is an orthomodular poset and we consider the partial operation $\wedge$ on $P$ then clearly $(P,\wedge,1)$ is a partial commutative monoid by the previous definition. It is worth noticing that it can happen that e.g.\ $(x\wedge y)\wedge z$ is defined, but $x\wedge(y\wedge z)$ is not, see the following example.

\begin{example}
Consider the orthomodular poset $(P,\subseteq,{}',\emptyset,\{1,\ldots,6\})$ where $P$ denotes the set of all subsets of $\{1,\ldots,6\}$ of even cardinality. Then $(\{1,2\}\wedge\{1,2,3,4\})\wedge\{2,3,4,5\}$ exists, namely
\[
(\{1,2\}\wedge\{1,2,3,4\})\wedge\{2,3,4,5\}=\{1,2\}\wedge\{2,3,4,5\}=\emptyset,
\]
but $\{1,2,3,4\}\wedge\{2,3,4,5\}$ and hence $\{1,2\}\wedge(\{1,2,3,4\}\wedge\{2,3,4,5\})$ do not exist.
\end{example}

Now we are ready to define our main concept.

\begin{definition}
An {\em unsharp residuated poset} is an ordered seventuple $\mathbf R=(R,\leq,\odot,\rightarrow,{}',0,1)$ where $\rightarrow:R^2\rightarrow2^R$ and the following hold for all $x,y,z\in R$:
\begin{enumerate}[{\rm(R1)}]
\item $(R,\leq,{}',0,1)$ is a bounded poset with an antitone involution,
\item $(R,\odot,1)$ is a partial commutative monoid where $x\odot y$ is defined whenever $x'\leq y$ and where $x\odot z\leq y\odot z$ whenever $x\leq y$ and $x\odot z$ and $y\odot z$ are defined,
\item $U(x,y')\odot y\subseteq UL(y,z)$ if and only if $U(x,y')\subseteq U(y\rightarrow z)$,
\item $x\leq y$ implies $U(x'\rightarrow y)=U(y)$.
\end{enumerate}
Condition {\rm(R3)} is called {\em unsharp adjointness}. The unsharp residuated poset $\mathbf R$ is called {\em divisible} if it satisfies the identity $x\odot(x\rightarrow y)\approx L(x,y)$, and it is called {\em idempotent} if $x\odot x$ is defined for every $x\in R$ and $x\odot x=x$.
\end{definition}

Note that because of $y'\leq U(x,y')$ the expression $U(x,y')\odot y$ is everywhere defined. Here $y'\leq U(x,y')$ means that $y'\leq z$ for all $z\in U(x,y')$, and $U(x,y')\odot y$ denotes the set $\{z\odot y\mid z\in U(x,y')\}$. Analogously, we proceed in similar cases, e.g.\ $L(x,y)\leq x$, thus $x\rightarrow y=x'\vee L(x,y)$ is everywhere defined.

We are able to show that the partial operations $\odot$ and $\wedge$ coincide provided $\mathbf R$ is idempotent. The precise formulation is as follows.

\begin{lemma}\label{lem2}
Let $(R,\leq,\odot,\rightarrow,{}',0,1)$ be an idempotent unsharp residuated poset and $a,b\in R$ and assume that $a\odot b$ as well as $a\wedge b$ are defined. Then $a\odot b=a\wedge b$.
\end{lemma}

\begin{proof}
We have
\[
a\odot b=(a\odot b)\wedge(a\odot b)\leq(a\odot1)\wedge(1\odot b)=a\wedge b=(a\wedge b)\odot(a\wedge b)\leq a\odot b.
\]
\end{proof}

Let us remark that in the definition of unsharp adjointness we have an additional element $y$ in the term $L(y,z)$ on the right-hand side of the first inequality and an additional element $y'$ in the term $U(x,y')$ on the left-hand side of both inequalities. Such an approach was already used by the authors in \cite{CL1} under the name ``relative adjointness''.

Since for two subsets $A,B$ of a poset $(P,\leq)$, $A\subseteq U(B)$ is equivalent to $A\geq B$, we have that unsharp adjointness (R3) is equivalent to the following {\em dual unsharp adjointness}:
\begin{enumerate}
\item[(R3')] $U(x,y')\odot y\geq L(y,z)$ if and only if $U(x,y')\geq y\rightarrow z$.
\end{enumerate}

An example of an unsharp residuated poset is shown in the following

\begin{example}
Let $R$ denote the six-element set $\{0,a,a',b,b',1\}$ and put $\mathbf R:=(P,\leq,\odot,\rightarrow,{}',0,1)$ where $a,a',b,b'$ are atoms as well as coatoms of $(P,\leq)$ and where the operations $\odot$, $\rightarrow$ and $'$ are given by the following tables:
\[
\begin{array}{l|cccccl}
\odot & 0 & a & a' & b & b' & 1 \\
\hline
  0   & 0 & 0 & 0  & 0 & 0  & 0 \\
  a   & 0 & a & 0  & 0 & 0  & a \\
  a'  & 0 & 0 & a' & 0 & 0  & a' \\
  b   & 0 & 0 & 0  & b & 0  & b \\
  b'  & 0 & 0 & 0  & 0 & b' & b' \\
  1   & 0 & a & a' & b & b' & 1
\end{array}
\quad
\begin{array}{l|lccccc}
\rightarrow & 0 & a & a' & b & b' & 1 \\
\hline
0  & \{1\}  &  \{1\}   &  \{1\}   &  \{1\}   &  \{1\}   &  \{1\} \\
a  & \{a'\} & \{a',1\} &  \{a'\}  &  \{a'\}  &  \{a'\}  & \{a',1\} \\
a' & \{a\}  &  \{a\}   & \{a,1\}  &  \{a\}   &  \{a\}   & \{a,1\} \\
b  & \{b'\} &  \{b'\}  &  \{b'\}  & \{b',1\} &  \{b'\}  & \{b',1\} \\
b' & \{b\}  &  \{b\}   &  \{b\}   &  \{b\}   & \{b,1\}  & \{b,1\} \\
1  & \{0\}  & \{0,a\}  & \{0,a'\} & \{0,b\}  & \{0,b'\} &    P \\
\end{array}
\quad
\begin{array}{l|l}
x & x' \\
\hline
0 & 1 \\
a & a' \\
a' & a \\
b & b' \\
b' & b \\
1 & 0
\end{array}
\]
It is an easy exercise to check that $\mathbf R$ is a divisible idempotent unsharp residuated poset.
\end{example}

Using our concept of an unsharp residuated poset, we can prove the following conversion of an orthomodular poset into this kind of residuated poset.

\begin{theorem}\label{th1}
Let $\mathbf P=(P,\leq,{}',0,1)$ be an orthomodular poset and put
\begin{align*}
      x\odot y & :=x\wedge y\text{ whenever }x\wedge y\text{ is defined}, \\
x\rightarrow y & :=x'\vee L(x,y)
\end{align*}
{\rm(}$x,y\in P${\rm)}. Then $\mathbb R(\mathbf P):=(P,\leq,\odot,\rightarrow,{}',0,1)$ is a divisible idempotent unsharp residuated poset.
\end{theorem}

Note that because of $L(x,y)\leq x$ the expression $x'\vee L(x,y)$ is everywhere defined.

\begin{proof}[Proof of Theorem~\ref{th1}]
Let $a,b,c\in E$. Condition (R1) and the first part of condition (R2) are obvious, and the second part of condition (R2) follows directly by definition.
\begin{enumerate}
\item[(R3)] If $U(a,b')\odot b\subseteq UL(b,c)$ then $U(a,b')\wedge b\subseteq UL(b,c)$ and, by Lemma~\ref{lem1},
\[
U(a,b')=b'\vee(U(a,b')\wedge b)\subseteq U(b'\vee L(b,c))=U(b\rightarrow c).
\]
If, conversely, $U(a,b')\subseteq U(b\rightarrow c)$ then $U(a,b')\subseteq U(b'\vee L(b,c))$ and, using Lemma~\ref{lem1} 1 once more,
\[
U(a,b')\odot b=U(a,b')\wedge b\subseteq U((b'\vee L(b,c))\wedge b)=UL(b,c).
\]
\item[(R4)] If $a\leq b$ then $a'\wedge b$ and $a\vee(a'\wedge b)$ are defined and, by the orthomodular law,
\[
U(a'\rightarrow b)=U(a\vee L(a',b))=U(a\vee L(a'\wedge b))=U(a\vee(a'\wedge b))=U(b).
\]
\end{enumerate}
Since $x\wedge x=x$ holds for all $x\in P$, $\mathbb R(\mathbf P)$ is idempotent. Finally, by Lemma~\ref{lem1} we have $a\odot(a\rightarrow b)=a\wedge(a'\vee L(a,b))=L(a,b)$ showing divisibility of $\mathbb R(\mathbf P)$.
\end{proof}

Note that $\mathbb R(\mathbf P)$ satisfies the identity $x\odot0\approx0\odot x\approx0$.

That the concept of an unsharp residuated poset was chosen appropriate is witnessed by the fact that also conversely, under some additional assumptions, such a poset can be organized into an orthomodular poset.

\begin{theorem}\label{th2}
Let $\mathbf R=(R,\leq,\odot,\rightarrow,{}',0,1)$ be an idempotent unsharp residuated poset satisfying the following conditions for all $x,y\in R$:
\begin{enumerate}[{\rm(i)}]
\item If $x'\leq y$ then $x\wedge y$ is defined,
\item $x\rightarrow y=x'\vee L(x,y)$.
\end{enumerate}
Then $\mathbb P(\mathbf R):=(R,\leq,{}',0,1)$ is an orthomodular poset.
\end{theorem}

\begin{proof}
Let $a,b\in R$. Since $\mathbf R$ is idempotent, by Lemma~\ref{lem2} we have that $a\wedge b=a\odot b$ whenever $a'\leq b$. If $a\leq b$ then $a'\wedge b$ and $a\vee(a'\wedge b)$ are defined and, using (R4), we obtain
\[
U(b)=U(a'\rightarrow b)=U(a\vee L(a',b))=U(a\vee L(a'\wedge b))=U(a\vee(a'\wedge b)),
\]
i.e.\ $b=a\vee(a'\wedge b)$ proving orthomodularity. As mentioned above, $'$ is a complementation. Thus $\mathbb P(\mathbf R)$ is an orthomodular poset.
\end{proof}

Every orthomodular poset is determined by its associated unsharp residuated poset as the following theorem shows.

\begin{theorem}
Let $\mathbf P=(P,\leq,{}',0,1)$ be an orthomodular poset. Then $\mathbb P(\mathbb R(\mathbf P))=\mathbf P$.
\end{theorem}

\begin{proof}
Let $(P,\leq,\odot,\rightarrow,{}',0,1)$ denote the unsharp residuated poset $\mathbb R(\mathbf P)$ associated to $\mathbf P$. Then $\mathbb R(\mathbf P)$ has the same operation $'$, the same ordering $\leq$ and the same elements $0$ and $1$ as $\mathbf P$ and the two conditions of Theorem~\ref{th2} are satisfied. Thus we obtain
\[
\mathbb P(\mathbb R(\mathbf P))=\mathbb P(P,\leq,\odot,\rightarrow,{}',0,1)=(P,\leq,{}',0,1)=\mathbf P.
\]
\end{proof}

Conversely, we can prove

\begin{theorem}
Let $\mathbf R=(R,\leq,\odot,\rightarrow,{}',0,1)$ be an idempotent unsharp residuated poset satisfying conditions {\rm(i)} and {\rm(ii)} of Theorem~\ref{th2} for all $x,y\in R$. Then $\mathbb R(\mathbb P(\mathbf R))=(R,\leq,\otimes,\rightarrow,{}',0,1)$ is an unsharp residuated poset where $x\otimes y=x\odot y$ provided $x'\leq y$ {\rm(}$x,y\in R${\rm)}.
\end{theorem}

\begin{proof}
Let $a,b\in R$. According to Theorem~\ref{th2}, $\mathbb P(\mathbf R)=(R,\leq,{}',0,1)$ is an orthomodular poset. Let $(R,\leq,\otimes,\Rightarrow,{}',0,1)$ denote the unsharp residuated poset $\mathbb R(\mathbb P(\mathbf R))$. Then
\begin{align*}
    a\otimes b & =a\wedge b=a\odot b\text{ provided }a'\leq b, \\
a\Rightarrow b & =a'\vee L(a,b)=a\rightarrow b.
\end{align*}
\end{proof}

Compliance with Ethical Standards: This study was funded by \"OAD, project CZ~02/2019, and, concerning the first author, by IGA, project P\v rF~2019~015. The authors declare that they have no conflict of interest. This article does not contain any studies with human participants or animals performed by any of the authors.

Authors' addresses:

Ivan Chajda \\
Palack\'y University Olomouc \\
Faculty of Science \\
Department of Algebra and Geometry \\
17.\ listopadu 12 \\
771 46 Olomouc \\
Czech Republic \\
ivan.chajda@upol.cz

Helmut L\"anger \\
TU Wien \\
Faculty of Mathematics and Geoinformation \\
Institute of Discrete Mathematics and Geometry \\
Wiedner Hauptstra\ss e 8-10 \\
1040 Vienna \\
Austria, and \\
Palack\'y University Olomouc \\
Faculty of Science \\
Department of Algebra and Geometry \\
17.\ listopadu 12 \\
771 46 Olomouc \\
Czech Republic \\
helmut.laenger@tuwien.ac.at

\begin{thebibliography}9
\bibitem B
L.~Beran, Orthomodular Lattices. Algebraic Approach. D.~Reidel, Dordrecht 1985. ISBN 90-277-1715-X.
\bibitem{CL17a}
I.~Chajda and H.~L\"anger, Residuation in orthomodular lattices. Topol.\ Algebra Appl.\ {\bf5} (2017), 1--5.
\bibitem{CL17b}
I.~Chajda and H.~L\"anger, Orthomodular lattices can be converted into left residuated l-groupoids. Miskolc Math.\ Notes {\bf18} (2017), 685--689.
\bibitem{CL1}
I.~Chajda and H.~L\"anger, Relatively residuated lattices and posets.	Math.\ Slovaca (submitted). http://arxiv.org/abs/1901.06664.
\bibitem{CL2}
I.~Chajda and H.~L\"anger, Unsharp residuation in effect algebras. J.\ Multiple-Valued Logic Soft Computing (submitted). http://arxiv.org/abs/1907.02738.
\bibitem{GJKO}
N.~Galatos, P.~Jipsen, T.~Kowalski and H.~Ono, Residuated Lattices: An Algebraic Glimpse at Substructural Logics, Elsevier, Amsterdam 2007. ISBN 978-0-444-52141-5.
\end{thebibliography}
\end{document}